\documentclass[12pt]{amsart}
\usepackage{graphicx,tikz,url}
\usepackage{amssymb,amscd}
\usepackage{amsfonts,mathrsfs}
\usepackage{enumerate}
\usetikzlibrary{arrows, automata,chains,fit,shapes}
\usepackage{subfig,float}
\usetikzlibrary{decorations.markings}

\usepackage{array}
\usepackage{epstopdf}

\usepackage{xcolor}

\usepackage{oubraces}
\usepackage{mathdots}

\newcommand{\ra}{\rightarrow}

\newcommand\var{\mathsf{Var}}

\newcommand\Z{\mathbb Z}

\newcommand\RR{\mathcal R}

\newtheorem{thm}{Theorem}[section]
\newtheorem{prop}[thm]{Proposition}
\newtheorem{lem}[thm]{Lemma}
\newtheorem{claim}[thm]{Claim}
\newtheorem{cor}[thm]{Corollary}

\theoremstyle{definition}
\newtheorem{defn}[thm]{Definition}
\newtheorem{rmk}[thm]{Remark}

\newtheorem{example}[thm]{Example}

\begin{document}

\tikzset{->-/.style={decoration={
  markings,
  mark=at position #1 with {\arrow{>}}},postaction={decorate}}}

\title[Sub-dominant cogrowth behaviour and amenability]
{
Sub-dominant cogrowth behaviour
and the viability of deciding amenability 
numerically}

\author[M. Elder]{Murray Elder}
\address{School of Mathematical and Physical Sciences,
The University of Newcastle,
Callaghan NSW 2308, Australia}
\email{Murray.Elder@newcastle.edu.au}

\author[C. Rogers]{Cameron Rogers}
\address{School of Mathematical and Physical Sciences,
The University of Newcastle,
Callaghan NSW 2308, Australia}
\email{Cameron.Rogers@uon.edu.au}

\date{\today}
\thanks{Research supported by Australian Research Council grant  FT110100178}

\begin{abstract}
We critically analyse a recent numerical method due to the first author, Rechnitzer and van Rensburg, which attempts to detect amenability or non-amenability in a finitely generated group by numerically estimating its asymptotic cogrowth rate.
We identify two potential sources of error. We then propose a modification of the method that enables it to easily compute surprisingly accurate estimates for initial terms of the cogrowth sequence.
\end{abstract}

\maketitle

\section{Introduction}

Researchers  studying 
the amenability Thompson's group $F$ will be familiar with a distrust of experimental methods applied to this problem.  
Part of this scepticism stems from the fact that (if it is amenable) $F$ is known to have a very quickly growing \emph{F\o lner function} \cite{Moore-Folner}.
However, experimental algorithms investigating amenability are rarely based on F\o lner's criteria directly, and 
to date
no identification is made in the literature of a mechanism by which a quickly growing F\o lner function could interfere with a given experimental method. 

In this paper we identify such a mechanism for a recent algorithm proposed by first author, A. Rechnitzer, and E. 
J. Janse  van Rensburg \cite{ERR}, which was designed  to experimentally detect amenability via the Grigorchuk-Cohen 
characterisation  in terms 
of the cogrowth function. We will refer to this as the ERR algorithm in the sequel.

We show that, in the ERR algorithm, estimates of the asymptotic cogrowth rate 
are compromised by sub-dominant behaviour in the reduced-cogrowth function. 

However, 
even though sub-dominant behaviour in the cogrowth function may interfere with estimates of the asymptotic growth rate, the ERR algorithm  can still be used to estimate other properties of the cogrowth function to high levels of accuracy.
In particular we are able re-purpose the algorithm to quickly estimate initial values of the cogrowth function even for groups for which the determination of the asymptotic growth rate is not 
possible (for example groups with unsolvable word problem).

The present work  started out as an independent verification by the second author 
of the experimental results in \cite{ERR}, as part of his PhD research.
More details can be found in \cite{CamPhD}.

The article is organised as follows. 
In Section~\ref{sec:prelim} we give the necessary background on amenability, random walks and cogrowth, followed by a summary of previous experimental work on the amenability of $F$. In Section~\ref{sec:R} a function quantifying 
the sub-dominant properties of the reduced-cogrowth function is defined.
In Section~\ref{sec:ERRsection} the ERR algorithm is summarised, followed by an analysis of two types of pathological behaviour in Section~\ref{sec:pathological_behaviour}.  The first of these is easily handled, while the second  is shown to depend on sub-dominant terms in the reduced-cogrowth function. 
In Section~\ref{sec:appropriation} the ERR method is modified to provide estimates of initial cogrowth values.  Using this the first 2000 terms for the cogrowth function of Thompson's group $F$ are estimated.

\section{Preliminaries}\label{sec:prelim}

We begin with a definition of terms and a quick survey of experimental work done on estimating amenability. 

\subsection{Characterisations of amenability}

The following characterisation of amenability is due to Grigorchuk \cite{grigorchuk1980Cogrowth} and Cohen \cite{cohen1982cogrowth}.  A shorter proof of the equivalence of this criteria with amenability was provided by Szwarc \cite{Szwarc_on_grig_cohen}.
\begin{defn}
\label{defn:amenabilityCogrowth}
Let $G$ be a finitely generated non-free group with symmetric generating set $S$. 
 Let $c_n$ denote the number of freely reduced words of length $n$
over $S$
which are equal to the identity in $G$.
Then $G$ is amenable if and only if 
$$\limsup_{n\rightarrow\infty} c_n^{1/n}=|S|-1.$$

Equivalently, let $d_n$ denote the number of  words (reduced and unreduced) of length $n$ over $S$
which are equal to the identity. Then $G$ is amenable if and only if 
$$\limsup_{n\rightarrow\infty} d_n^{1/n}=|S|.$$

The function $n\mapsto c_n$ is called the {\em reduced-cogrowth function} for $G$ with respect to $S$, and $n\mapsto d_n$ the {\em cogrowth function}.
\end{defn}

Kesten's criteria for amenability is given in terms of the probability of a random walk on the group returning to 
its starting point.

\begin{defn}\label{defn:amenabilityKesten}
Let $G$ be a finitely generated group, and let $\mu$ be a symmetric 
measure on $G$.  The random walk motivated by $\mu$
is a Markov chain on the group starting at the identity where the probability of moving from $x$ to $y$ is $\mu(x^{-1}y)$.  
Note the distribution after $n$ steps is given by the $n$-fold 
convolution power of $\mu$, which we denote as $\mu_n$. That is, $\mu_n(g)$ is the probability that an $n$-step walk starting at $e$ ends at $g$.
By Kesten's criteria \cite{Kesten} a group is amenable 
if and only if $$\limsup_{n\rightarrow\infty} (\mu_n(e))^{1/n}=1.$$
\end{defn}

Pittet and Saloff-Coste  proved that the asymptotic 
decay rate of the probability of return function is independent  
of measure chosen, up to the usual equivalence \cite{stabilityRandomWalk}.
For finitely generated groups  we can choose the 
random walk motivated by the uniform probability measure on a finite generating set.  This random walk is called a \emph{simple 
random walk} and corresponds exactly with a random walk on the Cayley graph.
For this measure the probability of return is given by 
\begin{equation}\label{eqn:mu-d} 
\mu_n(e) =\frac{d_n}{|S|^n},\end{equation}
where the (reduced and non-reduced) cogrowth terms $d_n$ are calculated with 
respect to the support of the measure.
Thus the cogrowth
function arises from a special case of return probabilities.

F\o lner's characterisation of amenability \cite{Folner} can be phrased in several 
ways.  Here we give the definition for finitely generated
groups.

\begin{defn}
\label{defn:amenabilityFolner}
Let $G$ be a group with finite generating set $S$.  For each 
finite subset $F\subseteq G$, we denote by $|F|$ the number of 
elements in $F$.  The {\em boundary} of a finite set $F$ is defined to be 
$$\partial F=\lbrace
g\in G\;:\;g\notin F,  gs\in F \text{ for some }s\in S
\rbrace.$$
A finitely generated group $G$ is amenable if and only if there exists 
a sequence of finite subsets $F_n$  such that 
$$
\lim_{n\rightarrow\infty} \frac{\vert \partial F_n \vert}{\vert F_n\vert}
=0.$$
\end{defn}

 Vershik \cite{Vershik-folner-function} defined the following function as a way to quantify how much of the Cayley graph must be considered before sets with a given isoperimetric profile can be found.

\begin{defn}
The F\o lner function of a group is 
$$f(n)=\min\left\lbrace |F|\;:\;\frac{|\partial F|}{|F|}<\frac{1}{n}  \right\rbrace.$$
\end{defn}

Significant literature exists on F\o lner functions.  It is known that 
there exists finitely 
presented amenable groups with F\o lner functions 
growing faster than $n^{n^n}$
(\cite{KrophollerMartino} Corollary~6.3)
and finitely generated groups  (iterated wreath product of  $k$ copies of $\Z$) with F\o lner functions 
growing faster than $\displaystyle n^{n^{\iddots}}$ of height $k$ for arbitrary $k$  \cite{ershler2003isoperimetric}.

\subsection{Experimental work on the amenability of $F$}
Richard Thompson's group $F$ is the group with 
 presentation
 \begin{equation}
 \langle a,b \mid  [ab^{-1},a^{-1}ba],[ab^{-1},a^{-2}ba^2] \rangle
\label{eqn:Fpresentation}
\end{equation}
where $[x,y]=xyx^{-1}y^{-1}$ denotes the commutator of two elements. See for example \cite{CFP} for a more detailed introduction to this group.

Whether or not $F$ is amenable has attracted a large amount of interest, and has so far evaded many different attempts at a proof of both positive and negative answers.

The following is a short summary of experimental work previously done on Thompson's group $F$. 

\begin{itemize}

\item[\cite{ComputationalExplorationsF}]
  Burillo, Cleary and  Wiest 2007.
 The authors randomly choose words and reduce them to a normal form to test if they represent the identity element.  From this they estimate the proportion of words of length $n$  equal to the identity, as a way to compute the asymptotic growth rate of the cogrowth function.

\item[\cite{Arzhantseva}]
  Arzhantseva,  Guba, Lustig, and Pr{\'e}aux 2008.
The authors study the {\em density} or  least upper bound for the average vertex degree  of any finite subgraph of the Cayley graph; an $m$-generated group is amenable if and only if the density of the corresponding Cayley graph is $2m$ (considering inverse edges as distinct). A computer program is run and data is collected on a range of amenable and non-amenable groups. They find a finite subset 
in  $F$ with density $2.89577$ with respect to the $2$ generator presentation above. (To be amenable one would need to find sets whose density approaches $4$). 
Subsequent theoretical work of Belk and Brown gives sets with density approaching $3.5$ \cite{BelkBrown}.

\item[\cite{ElderCogrowthofThompsons}] 
Elder, Rechnitzer and Wong 2012.
Lower bounds on the cogrowth rates of various groups are  obtained by computing the dominant eigenvalue of the adjacency matrix of truncated Cayley graphs. These bounds are  extrapolated to estimate the cogrowth rate.
As a byproduct the first 22 coefficients  of the cogrowth  series are computed exactly.

\item[\cite{Haagerup}]  
Haagerup, Haagerup, and Ramirez-Solano 2015.
Precise lower bounds of certain norms of elements in the group ring of $F$ are computed, and 
	coefficients  of the first 48 terms of the  cogrowth  series are computed exactly.

\item[\cite{ERR}] 
Elder, Rechnitzer and van Rensburg 2015.
The {\em Metropolis Monte Carlo} method from statistical mechanics is adapted to estimate the asymptotic growth rate of the cogrowth function by running  random walks on the set of all trivial words in a group. The results obtained for Thompson's group $F$ suggest it to be non-amenable.
We describe their method in more detail in Section~\ref{sec:ERRsection} below.

\end{itemize}

Justin Moore  \cite{Moore-Folner} (2013) has shown that if $F$ were amenable  then its F\o lner function  would increase
faster than a tower of $n-1$ twos, 
$$2^{2^{2^{\iddots}}}$$ 
This result has been proposed as an obstruction to all computational methods for approximating amenability; a computationally infeasibly large portion of the Cayley graph must be considered before sets with small boundaries can be found.
However, in all but one of the experimental algorithms listed above computing F\o lner sets was not the principle aim.
In order to understand how a bad F\o lner function affects the performance of these methods, we need to understand the connection between convergence properties of the respective limits in the various characterisations of amenability.

\section{Quantifying sub-dominant cogrowth behaviour}
\label{sec:R}
The F\o lner function 
quantifies the rate of convergence of the limit in Definition~\ref{defn:amenabilityFolner}.  We consider the following definitions as an attempt to quantify the rate of convergence of the limits in Definition~\ref{defn:amenabilityCogrowth}.

\begin{defn}\label{defn:R}
Let $G$ be a finitely generated group with symmetric generating set 
$S$.
 Let 
$c_n$ be the number of all reduced
trivial words of length $n$ and let $C=\limsup c_n^{1/n}.$
Define
$$\RR(n)=\min
\left\lbrace k \;:\;\frac{c_{2k+2}}{c_{2k}}>C^2-\frac{1}{n} 
\right\rbrace$$
\end{defn}

Definition \ref{defn:R} uses only even word lengths (and hence $C^2$ instead of $C$).  This is necessary because group presentations with only even length relators have no odd length trivial words.
For this paper we will only consider the function $\RR$ for amenable groups, in which case $C=|S|-1$ except when the group is free (infinite cyclic).

A similar definition may be made for the cogrowth function. 
\begin{defn}\label{defn:Rprime}
For $G$ a finitely generated group with symmetric generating set 
$S$ we may define 
$$\RR'(n)=\min
\left\lbrace k \;:\;\frac{d_{2k+2}}{d_{2k}}>D^2-\frac{1}{n} 
\right\rbrace$$
where
$d_n$ be the number of all (reduced and non-reduced)
trivial words of length $n$ and $D=\limsup c_n^{1/n}.$
\end{defn}

Literature already exists studying the convergence properties of return probabilities, and we suspect that
the function $\RR'$ is a reformulation of the {\em $L^2$-isoperimetric 
function}  \cite{BendikovPittetSauer}.  

\begin{example}
For the trivial group with some finite symmetric generating set $S$ we have $c_0=1, c_k=|S|(|S|-1)^{k-1}$ for $k\geq 1$ so  
$\frac{c_{2k+2}}{c_{2k}}\geq (|S|-1)^2$ and $\RR(n)=0$. 
Similarly since $d_k=|S|^k$ we have
 $\RR(n)=\RR'(n)=0$. 
\end{example}

Aside from the trivial group, it is usually easier to compute $\RR'$ (or its asymptotics) than it is to obtain $\RR$. 
For this reason we first consider $\RR'$ functions for various groups, and then prove that for infinite, amenable, non-free  groups $\RR'$ and $\RR$ have the same asymptotic behaviour.

\begin{example}
For any finite group 
 the rate of growth of $d_n$ is the dominant eigenvalue of the adjacency matrix of the 
 Cayley graph, and some simple analysis shows that $\RR'(n)$ is at most logarithmic in $n$. 
\end{example}

Define $f\precsim g$  if there exist constants $a, b > 0$, such that for $x$ large enough, $f(x) \leq ag(bx)$. Then $f\sim g$ ($f$ and $g$ are asymptotic) if $f\precsim g$ and $g\precsim f$.

Table \ref{tab:differentRn} provides a  sample of amenable groups for which  the asymptotics of $\RR'(n)$,  the F\o lner function and probabilities of return are known  \cite{ershler2003isoperimetric,randomWalkWreathProducts,PittetSCsolvable2003}.

\begin{table}
\begin{center}

\renewcommand{\arraystretch}{2.5}
\begin{tabular}
{
|>{\centering\arraybackslash}p{2.7cm}
|>{\centering\arraybackslash}p{2.7cm}
|>{\centering\arraybackslash}p{3.3cm}
|>{\centering\arraybackslash}p{2.7cm}
|}
\hline
 Example & $\mathcal{F}(n)$ & $\mu_n (e)$  & $\mathcal{R}'(n)$  \\ 
\hline \hline
 trivial  & $\sim$ constant & $\sim$ constant & $\sim$ constant \\ 
\hline 
 $\Z^k$ & $\sim n^{k}$ & $\sim n^{-k/2}$  &  $\sim n$   \\ 
\hline
 $BS(1,N)$ & $\sim e^n$ & $\sim e^{-n^{1/3}}$
& $\sim n^{3/2}$ \\
\hline
 $\Z\wr\Z$  & $n^n$ & $\sim e^{-n^{1/3}(\ln n)^{2/3}}$ & $\sim\ln(n) n^{3/2}$\\
 \hline
  $\Z\wr\Z\wr\dots \wr\Z$ $(d-1)$-fold wreath product  & $n^{n^{n^{\iddots^n}}}$ (tower of $d-1$ $n$'s) & $\sim e^{-n^{\frac{d}{d+2}}(\ln n)^{\frac{2}{d+2}}}$ & $\sim\ln(n) n^{(d+2)/2}$\\
 \hline
 \end{tabular} 
\end{center}
\caption{Comparing asymptotics of the probabilities of return, the F\o lner function $\mathcal{F}$, and $\RR'$ for various 
groups. 
\label{tab:differentRn}}
\end{table}

The results for the asymptotics of $\RR'(n)$ were derived directly from the known  asymptotics for  $\mu_n$. A discussion of these methods  will appear in \cite{CamPhD}. In practice however it proved quicker to guess the asymptotics and then refine using the following method.

\begin{prop}\label{prop:ProvingRn}
The  asymptotic results for $\RR'(n)$ in Table \ref{tab:differentRn} are correct.
\end{prop}

\begin{proof} 
For a given group suppose
$\mu_n(e)\sim g(n)$ where $g$ is a continuous real valued function, as in Table \ref{tab:differentRn}.  
Then $d_n\sim |S|^n g(n)$.

Finding $\RR'(n)$ requires solving the equation 
\begin{equation}\frac{d_{2k+2}}{d_{2k}}=|S|^2-\frac{1}{n}
\label{eqn:Rnkn}
\end{equation} for $k=k(n)$. 
This is equivalent to solving
$$1=n\left(
|S|^2-\frac{d_{2k+2}}{d_{2k}}
\right)$$
for $k$.

Suppose $f(n)$ is a function where 
\begin{equation}\label{eqn:methodRn}
L  =\lim_{n\ra\infty}
n
\left(
|S|^2-\frac{d_{2f(n)+2}}{d_{2f(n)} }
\right) 
\end{equation}
exists and is non-zero. 

If $L=1$ then 
$$\left(
|S|^2-\frac{d_{2f(n)+2}}{d_{2f(n)} }
\right) \sim \frac{1}{n}$$ 
and so 
$$\frac{d_{2f(n)+2}}{d_{2f(n)}} \sim |S|^2-\frac{1}{n}.$$
Then
$k(n)\sim f(n)$ satisfies Equation \ref{eqn:Rnkn}.  Therefore $\RR'(n)$ is asymptotic to $f(n).$

If $L$ exists and is non-zero then 
$$\left(
|S|^2-\frac{d_{2f(n)+2}}{d_{2f(n)} }
\right) \sim \frac{L}{n}.$$
Then $$\left(
|S|^2-\frac{d_{2f(Ln)+2}}{d_{2f(Ln)} }
\right) \sim \frac{L}{Ln}=\frac{1}{n}$$
and so
$\RR'(n)\sim f(L n)$.

The derivations of candidates for $f(n)$ in each case in Table \ref{tab:differentRn} is performed in \cite{CamPhD}.  The results in the table do not include the constant $L$ since the probabilities of return used as input are only correct up to scaling. 
We leave the calculation of Equation \ref{eqn:methodRn} for the results from
Table \ref{tab:differentRn} as an exercise.
\end{proof}

\subsection{Converting from cogrowth to reduced-cogrowth}
We now prove an equivalence between the sub-dominant behaviour of 
the cogrowth and reduced-cogrowth functions. This allows us to borrow the previously listed results for $\RR'$ when discussing $\RR$ and the ERR method. 
The dominant and sub-dominant cogrowth behaviour can be 
analysed from the generating functions for these sequences.

\begin{defn}
Let $d_n$ denote the number of trivial words of length $n$ in 
a finitely generated group.  The \emph{cogrowth series} is 
defined to be $$D(z)=\sum_{n=0}^\infty d_n z^n.$$

Let $c_n$ denote the number of reduced trivial words.  Then 
$$C(z)=\sum_{n=0}^\infty c_n z^n$$ is said to be the 
\emph{reduced-cogrowth series}.
\end{defn}

$D$ and $C$ are the generating functions for $d_n$ and $c_n$ respectively, and are 
related in the following way.
 Let $|S|=2p$ be the size of a symmetric generating set.
Then from \cite{KouksovRationalCogrowth,cogrowthConvertWoess}
\begin{equation}
C(z)=\frac{1-z^2}{1+(2p-1)z^2}D\left(
\frac{z}{1+(2p-1)z^2}
\right)
\label{eqn:cFromD}
\end{equation}
and
\begin{equation}
D(z)=\frac{1-p+p\sqrt{1-4(2p-1)z^2}}{1-4p^2z^2}
C\left(
\frac{1-\sqrt{1-4(2p-1)z^2}}{2(2p-1)z}
\right).
\label{eqn:dFromC}
\end{equation}

The dominant and sub-dominant growth properties of the cogrowth functions may be analysed by considering the 
singularities of these generating functions.
For a detailed study of the relationship between singularities 
of generating functions 
and sub-dominant behaviours of coefficients see \cite{flajolet2009analytic}.

We now outline an example of how the composition of functions (as in Equations~\ref{eqn:cFromD} and \ref{eqn:dFromC}) effects 
the growth properties of the series coefficients.

\begin{example}\label{ex:CvsD}
Consider $$f(z)=\left(
1-\frac{z}{r}
\right)^{-p}.$$

Then (for positive $p$) $f(z)$ has a singularity at $z=r$, and this defines the 
radius of convergence of $f(z)$ and the asymptotic
growth rate of the series coefficients of the 
expansion of $f(z)$.  It 
also determines the principle sub-dominant term contributing 
to the growth of the coefficients.
In this example, the coefficients will grow like $ n^{p-1}r^{-n}.$

We wish to investigate what happens to this growth behaviour 
when we compose the function $f$ with a function $g$.
Consider $f(g(z))$ for some function $g$ for which $g(0)=0$.
The singularities of $g$ are inherited by $f(g(z))$; if $g$ is 
analytic everywhere then the only singularities of $f(g(z))$ 
will occur when $g(z)=r$. In this case, the new radius of convergence
will be the minimum $|z|$ such that $g(z)=r$. Importantly, however, the principle sub-dominant growth term of the
series coefficients will remain polynomial of degree $p-1$.

A variation on this behaviour will occur if there is an 
$r_0$ for which $g(z)$ is 
analytic on the ball of radius $r_0$, and $g(z)=r$ for some 
$z$ in this region. Again, when this occurs, the new radius of 
convergence is obtained by solving $g(z)=r$ and the type 
of the principle sub-dominant term in the growth of the
coefficients remains unchanged.

If there does not exist such an $r_0$, the principle singularity
of $g(z)$ will dominate the growth properties of the
coefficients.
\end{example}

\begin{prop}\label{prop:nonFreeCvsD}
Let $G$ be an {infinite} amenable group generated by $p$ elements and their inverses. 
Then the principle sub-dominant terms contributing to the growth of $d_n$ and $c_n$ are asymptotically equivalent, except when the group is infinite cyclic.
\end{prop}

\begin{proof}
For an amenable group generated by $p$ elements and their inverses the radius of convergence for $D(z)$ is exactly
$1/2p$.  This follows immediately from Definition \ref{defn:amenabilityCogrowth}.

Now from Equation~\ref{eqn:cFromD}, the reduced-cogrowth 
series is obtained by composing the cogrowth series with 
$$p(z)=\frac{z}{1+(2p-1)z^2}$$ and then multiplying by 
$$q(z)=\frac{1-z^2}{1+(2p-1)z^2}.$$
Both of these functions are analytic inside the ball of radius 
$1/\sqrt{2p-1}$.

Now 
\begin{equation}\label{eq:alternateCogrowthProof}
p\left(\frac{1}{2p-1}\right)=\frac{1}{2p},
\end{equation}
the singularity of $D(z)$.  Hence, $1/(2p-1)$ is a singularity of $D(p(z))$, and hence of $C(z)$.  Note that if the group is infinite cyclic, then $p=1$ and $1/(2p-1)$ and $1/\sqrt{2p-1}$ are equal.  In this scenario the radius of convergence of $p(z)$ is reached at the same moment that 
$p(z)$ reaches the radius of convergence of $D(z)$. This means that both $p$ and $q$ contribute to the principle singularity, and this explains why the reduced and non-reduced cogrowth functions
for the infinite cyclic group exhibit such different behaviour.

If $p>1$ then  $1/(2p-1)$
is inside the ball of radius $1/\sqrt{2p-1}$ (ie, inside the region of convergence for $p$ and $q$).  Thus, the singularity of $D$ is reached before $z$ approaches the singularity of $p$ and $q.$ 

In this case the substitutions in Equation~\ref{eqn:cFromD} change the location of the principle singularity, but do not change the type of the singularity, or the form of the principle  sub-dominant 
term contributing to the growth of the series coefficients.
\end{proof}

\begin{cor}\label{cor:RvsRr}
Suppose $G$ is a finitely generated, infinite amenable group that is not the infinite cyclic group. Then $\RR$ is 
asymptotically 
equivalent to $\RR'$. 
\end{cor} 

\begin{rmk}
An alternate proof of the Grigorchuk/Cohen characterisation of amenability is easily constructed from an analysis of the singularities of $C(z)$ and $D(z)$.  For example, Equation \ref{eq:alternateCogrowthProof} proves the first result from Definition \ref{defn:amenabilityCogrowth}.
This argument also picks up that the infinite cyclic group presents a special case.  Though amenable, $\limsup_{n\ra\infty}c_n\neq |S|-1$.
For this group we have $\RR(n)\sim 0$ while $\RR'(n)\sim n$.
\end{rmk}

\subsection{Sub-dominant behaviour in the cogrowth of $F$
\label{sec:subDomInF}
}
The groups $BS(1,N)$ limit to $\Z\wr\Z$ in the space of marked groups.  This implies that the growth of the function $\mathcal{R}'$ and hence $\RR$ for $BS(1,N)$ increases with $N$.  This is consistent with Table \ref{tab:differentRn}, since these results do not include scaling constants.  This leads to the following result.

\begin{prop}\label{prop:connectBStoThompsons}
If Thompson's group $F$ is amenable, its $\mathcal{R}$ function grows faster than the $\mathcal{R}$  function for any $BS(1,N)$.  In particular, it is asymptotically super-polynomial.
\end{prop}

\begin{proof}
By the convergence of $BS(1,N)$ to $\Z\wr\Z$ in the space of marked groups we have that, for any $N$, the function $\RR'$ for $BS(1,N)$ grows slower than the corresponding function for $\Z\wr\Z$.  In \cite{stabilityRandomWalk} it is proved that, for finitely generated groups, the probability of return cannot asymptotically exceed the probability of return of any finitely generated subgroup. This implies that, for finitely generated amenable groups, the $\RR'$ function of the group must grow faster than the $\RR'$ function of any finitely generated subgroup.  Since there is a subgroup of $F$ isomorphic to
$\Z\wr\Z$, $\RR'(n)$ for $F$ must grow faster than $\RR'(n)$ for $\Z\wr\Z$ and hence $BS(1,N)$.

Since $F$ contains every finite depth iterated wreath products of $Z$ (\cite{GubaSapir} Corollary 20), 
the probability of return for $F$ decays faster than
$$e^{-n^{\frac{d}{d+2}}(\ln n)^{\frac{2}{d+2}}}$$ for any $d$. 
Taking the limit as $d$ approaches infinity of the corresponding values for $\RR'$ and then doing the conversion from $\RR'$ to $\RR$ gives the final result.
\end{proof}

Note that if $F$ is non-amenable, then even though it still contains these subgroups, they do not affect the $\RR'$ function.  In this scenario it is still true that the return probability for $F$ decays faster than the interated wreath product, because $F$ would have  exponentially decaying return probability.  For non-amenable groups the return probability does not identify the principle sub-dominant term in $d_n$, and hence does not correlate directly with $\RR'$.
\section{The ERR algorithm}
\label{sec:ERRsection}

We start  by summarising the original work
by the first author, Rechnitzer and van Rensburg. Only the details 
directly pertinent to the present paper are discussed here, for a more
detailed analysis of the random walk 
algorithm and a derivation 
of the stationary distribution we refer the 
reader to \cite{ERR}.  For the sake of brevity the 
random walk performed by the algorithm will be referred to as the ERR random walk.

Recall that a group presentation, denoted $\langle S \mid R \rangle$, consists
of a set $S$ of formal symbols (the generators) and a set $R$ of words written 
in $S^{\pm 1}$ (the relators) and corresponds to the quotient 
of the free group on $S$ by the normal closure of the 
relators $R$. In our paper, as in \cite{ERR}, all groups
will be finitely presented: both $S$ and $R$ will be finite.
Furthermore, the implementation of the algorithm assumes  both 
$S$ and $R$ to be symmetric, that is, $S=S^{-1}$ and $R=R^{-1}$.
In addition, for convenience $R$ is enlarged to be closed under cyclic permutation.
Recall that 
$c_n$ counts the 
number of reduced words in $S$ of length $n$ which represent
the identity in the group (that is, belong to the normal 
closure of $R$ in the free group).

\subsection{The ERR random walk}

The ERR random walk is not a random walk on the Cayley graph of a group, but instead a random walk on 
the set of trivial words for the group presentation. 
This makes the algorithm extremely easy to implement,  since it does not require an easily computable normal
 form or even a solution to the word problem.
The walk begins at the empty word, and  constructs
new 
trivial words from the current trivial word 
using one of two moves:
\begin{itemize}
	\item (conjugation by $x\in S$).  
	In this move an element is chosen from 
	$S$ according to a predetermined probability distribution.
	The current word is conjugated
	by the chosen generator and then freely reduced to produce 
	the new 	candidate word.
	\item (insertion of a relator).  In this move a relator is 
	chosen from $R$ according to a predetermined 
	distribution and inserted into the current word at a position
	chosen uniformly at random .  In order to maintain the detailed
	balance criteria (from which the stationary distribution 
	is derived) 
	it is necessary to allow only those insertions which 
	can be immediately reversed by inserting the inverse of the 
	relator at the same position. To this end the a notion of \emph{left insertion} is introduced; 
	after relators are inserted free reduction is done on only the left hand side of the relator. 
	If after this the word is not freely reduced the move is rejected.
\end{itemize}

Transition probabilities are defined which determine whether 
or not the trivial word created with these moves is accepted as the new state. These probabilities involve parameters $\alpha\in\mathbb{R}$ and 
$\beta\in (0,1)$ which may be adjusted
to control the distribution of the walk.

Let the current word be $w$ and  the candidate word be $w'$.

\begin{itemize}
\item If $w'$ was obtained from $w$ via a conjugation it is accepted as the new 
current state with probability 
$$\min \left\lbrace 1,
\left(\frac{\left\vert w'\right\vert+1}
{\left\vert w\right\vert+1}\right)^{1+\alpha} 
\beta^{\left\vert w'\right\vert-\left\vert w\right\vert}
\right\rbrace.$$

\item If $w'$ was obtained from $w$ via an insertion it  is accepted as the new 
state with probability 
$$\min \left\lbrace 1,
\left(\frac{\left\vert w'\right\vert+1}
{\left\vert w\right\vert+1}\right)^{\alpha} 
\beta^{\left\vert w'\right\vert-\left\vert w\right\vert}
\right\rbrace.$$
If $w'$ is not accepted the new 
state remains as $w$.
\end{itemize}

These probabilities are chosen so that the distribution 
on the set of all trivial words given by 
$$\pi\left(u\right)=\frac{\left(\left|u\right|+1\right)^{1+\alpha}
\beta^{\left|u\right|}}
{Z},$$
(where $Z$ is a normalizing constant) can be proved to be  the unique stationary 
distribution of the Markov chain, and the limiting 
distribution of the random walk.

The following result is then given.
\begin{prop}[\cite{ERR}]
 As $\beta$ approaches $$\beta_c = \frac1{\limsup_{n\rightarrow\infty} (c_n)^{1/n}}$$ 
 the expected value
of the word lengths visited approaches infinity. 
\end{prop}

This result leads to the following method for estimating the value of $\beta_c$. For each presentation, 
random walks are run with different values of $\beta$.  Average
word length is plotted against $\beta$.  The results obtained for Thompson's group $F$ are reproduced in Figure \ref{fig:ERRpaperThompsons}. The values for $\beta$ at which the data points diverge gives an indication of
$\beta_c$, and hence the amenability or otherwise of the group.

\begin{figure}
\includegraphics[width=110mm]{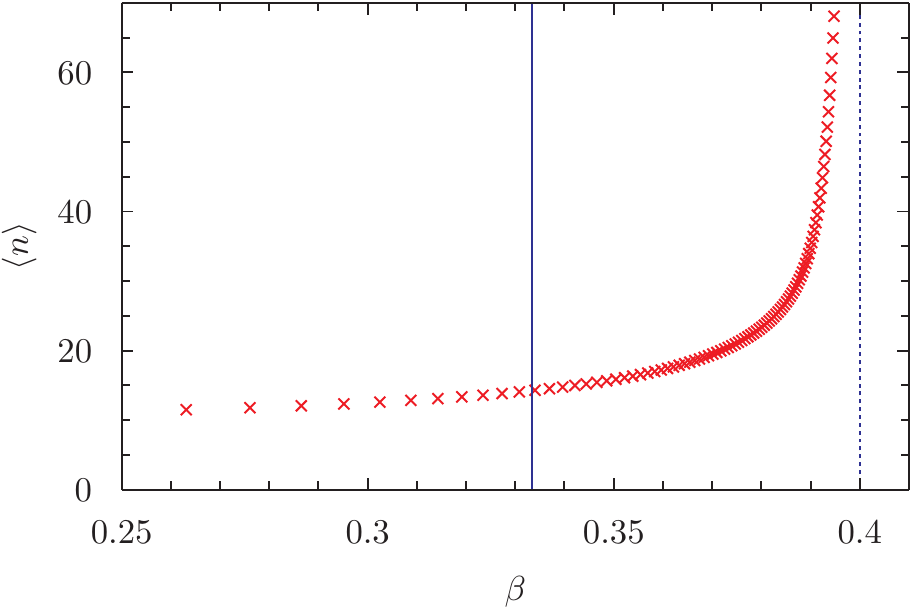} 
\caption{The results from \cite{ERR} of the ERR algorithm applied to the standard presentation of Thompson's group $F$. Each data point plots the average word length of an ERR random walk against the paramater $\beta$ used.
\label{fig:ERRpaperThompsons}
}
\end{figure}

Random walks were run on presentations for a selection of  amenable and non-amenable groups, including Baumslag-Solitar groups, some free product examples whose cogrowth series are known \cite{KouksovFreeProduct}, the genus 2 hyperbolic surface group, a finitely presented group related to the basilica group, and Thompson's group $F$.

The data in Figure \ref{fig:ERRpaperThompsons} appears to show fairly convincingly that the location of $\beta_c$ 
is a long way from the value of $\frac13$ expected were the group  amenable.

It is  noted in \cite{ERR}
that a long 
random walk may be split into shorter segments,
and the variation in average word lengths of the segments gives an 
estimation of the error in the estimated expected word length.

\begin{rmk}\label{rmk:implementation_details}
In the original work reported in \cite{ERR}, the algorithm was coded in {c++}, words were stored as linked lists,
 the GNU Scientific Library was used to generate pseudo-random numbers, and {\em parallel tempering} was
  used to speed up the convergence
of the random walk. For independent verification the second author
 coded the algorithm in python, kept words as strings, used the python package \emph{random}, and no 
 tempering was used. Results obtained were consistent with those in \cite{ERR}. The experimental analysis
  and modifications described in this paper use the python version of the second author.
\end{rmk}

\section{Investigating Pathological Behaviour\label{sec:pathological_behaviour}}

The theory underpinning the ERR random walk is 
complete --- the random walk is certain to converge to the stationary 
distribution.  This does not preclude, however, convergence 
happening at a computationally
indetectible rate.  
Since there are finitely presented groups with unsolvable word problem, there is no chance of deriving  bounds on the 
rates of convergence of the walk  in 
any generality. 
 In the process of independently verifying the results in \cite{ERR}, however, we were able to identify two properties of 
group presentations which appear to slow the rate of convergence.
The first of these is unconnected with the F\o lner function, and
does not pose any problem to the implementation of the ERR algorithm to Thompson's group $F$. 
It does, however, refute the claim in (\cite{ERR} Section 3.7) that the method can be successfully applied to infinite presentations.

\subsection{Walking on the wrong group}\label{subsec:wrong_group}
It is easy to see from the probabilistic selection criteria 
used by the ERR random walk that moves which increase the word 
length by a large amount are rejected with high probability. 
This poses a problem for group presentations containing long relators since insertion moves that attempt to insert a long relator will be 
accepted much less often than moves which attempt to insert a shorter relation.

The following example makes this explicit.

\begin{lem}
All presentations of the form 
$$\left\langle a,b\mid abab^{-1}a^{-1}b^{-1},\;a^n b^{-n-1}\right\rangle$$ 
describe the trivial group.
\end{lem}
\begin{proof}
Since $a^n=b^{n+1}$ we have $a^nba=   b^{n+1}ba=bb^{n+1}a=ba^{n+1}=bab^{n+1}.$

Since $aba=bab$ we have 
$a^iba=a^{i-1}bab$ so $a^nba=a^{n-1}bab=a^{n-2}bab^2=\dots =bab^{n}$.

Putting these results together gives $bab^n=bab^{n+1}$ and hence $b$ is trivial.   The result follows.
\end{proof}

By increasing $n$ we can make the second relator arbitrarily large 
without affecting the group represented by the presentation, or the 
group elements represented by the generators.  This implies that 
ERR random walks for each of these presentations should converge
to the same stationary distribution.

Changing the presentation, however, does change the number of 
steps in the ERR random walk needed to reach certain trivial words (such as the word `$a$').  

ERR random walks were performed on these presentations for 
$n= 1, 2, \dots,19$.  As well as recording the average 
word length of words visited, the number of \emph{accepted}
insertions of each relator was  recorded.

\begin{table}
\begin{center}
\begin{tabular}
{|>{\centering\arraybackslash}p{1cm}|>{\centering\arraybackslash}p{2cm}|>
{\centering\arraybackslash}p{3cm}|>{\centering\arraybackslash}p{3cm}|}
\hline 
$n$ & number of steps & number of accepted insertions of small
relator & number of accepted insertions for big relator \\ 
\hline 
1 & $2.0\times 10^8$ & $2977228$ & $7022772$ \\ 
\hline 
2 & $3.6\times 10^8$ & $4420185$ & $5579815$ \\ 
\hline 
3 & $6.1\times 10^8$ & $6323376$ & $3676624$ \\ 
\hline 
4 & $9.0\times 10^8$ & $8016495$ & $1983505$ \\ 
\hline 
5 & $1.2\times 10^9$ & $9088706$ & $911294$ \\ 
\hline 
6 & $1.4\times 10^9$ & $9621402$ & $378598$ \\ 
\hline 
7 & $1.5\times 10^9$ & $9850251$ & $149749$ \\ 
\hline 
8 & $1.7\times 10^9$ & $9943619$ & $56381$ \\ 
\hline 
9 & $1.8\times 10^9$ & $9977803$ & $22197$ \\ 
\hline 
10 & $1.9\times 10^9$ & $9991680$ & $8320$ \\ 
\hline 
11 & $2.1\times 10^9$ & $9997122$ & $2878$ \\ 
\hline 
12 & $2.2\times 10^9$ & $9998720$ & $1280$ \\ 
\hline 
13 & $2.2\times 10^9$ & $9999585$ & $415$ \\ 
\hline 
14 & $2.3\times 10^9$ & $9999938$ & $62$ \\ 
\hline 
15 & $2.4\times 10^9$ & $10000000$ & $0$ \\ 
\hline 
16 & $2.6\times 10^9$ & $10000000$ & $0$ \\ 
\hline 
17 & $2.7\times 10^9$ & $10000000$ & $0$ \\ 
\hline 
18 & $2.8\times 10^9$ & $10000000$ & $0$ \\ 
\hline 
19 & $2.9\times 10^9$ & $10000000$ & $0$ \\ 
\hline 
\end{tabular} 

\end{center}

\caption{\label{tab:trivial_relator_acceptences}  The ERR
algorithm applied to the trivial group with presentation 
$\left\langle a,b \mid aba=bab,\;a^n=b^{n+1} \right\rangle$ for various $n$.  As $n$ increases, the longer relator is successfully 
inserted less frequently.}
\end{table}

Table \ref{tab:trivial_relator_acceptences} shows the sharp decline in the number
of accepted insertions of the second relator as $n$ increases.
Indeed, for $n>14$ there were no instances in which the longer relator
was successfully inserted.  Unsurprisingly, walks for large $n$ did not converge to the same distribution as those where $n$ was small, and for large $n$ the data did not accurately predict the asymptotic growth rate of the cogrowth function. For these $n$ the ERR random walk was actually
taking place on $\langle a,b\mid abab^{-1}a^{-1}b^{-1} \rangle$, which is a presentation for the 3-stand braid group, which is non-amenable.

Note that, given enough time, the longer relator would be 
successfully sampled, and that an infinite random walk is still
guaranteed to converge to the theoretical distribution for the trivial group.  
Such convergence,
however, may take a computationally infeasible amount of time.

\begin{claim}
The presence of long relators in the input presentation slows
the rate at which an ERR random walk converges to the stationary distribution. Therefore, the ERR method cannot be reliably extended to accept infinite
presentations.
\end{claim}

This result is not surprising.
In \cite{BenliGrigHarpe} an infinitely presented 
amenable group is given for which any truncated presentation 
(removing all but a finite number of relators) is non-amenable.
The ERR method could not expect to succeed on this group even if 
long relators were sampled often; since the ERR random walk can only be run for a finite time there can 
never be a representative sampling of an infinite set of relators, so  ERR would incorrectly conclude this group is non-amenable.

The pathological presentations of the trivial group
studied here form 
a sequence of presentations for amenable (trivial) groups which 
approach a non-amenable group in the space of marked groups.
The failure  of the ERR method to predict amenability for 
these groups suggests that one does not need 
particularly elaborate or large presentations to produce pathological behaviour.

However,  we remark that this behaviour is easily monitored.  In addition to counting  
 the number of attempted moves of the walk, one should record the relative number of successful insertions of each relator. 
 In the case of Thompson's group $F$ the two relators have similar lengths, and in our experiments both were sampled with comparable frequency.
 
Further analysis of this phenomena appears \cite{CamPhD}.

\subsection{Sub-dominant behaviour in cogrowth.\label{subsec:subdom}}
Recall that the solvable Baumslag-Solitar groups $BS(1,n)=\langle	a,t\mid tat^{-1}a^{-n}\rangle$
are the only two generator, single-relator, amenable groups \cite{OneRelatorAmenable}; for each of these groups $\beta_c=1/3$. 
In \cite{ERR} walks were run on $BS(1,1)=\Z^2,\;BS(1,2)$ and $BS(1,3)$ and for these groups 
the random walk behaved as predicted with divergence occurring at the moment when $\beta$ exceeded $\beta_c$. 
It may be surprising then to see the output of some ERR walks run $BS(1,7)$ shown in Figure \ref{fig:ERR-BS17}.

\begin{figure}
\includegraphics[width=110mm]{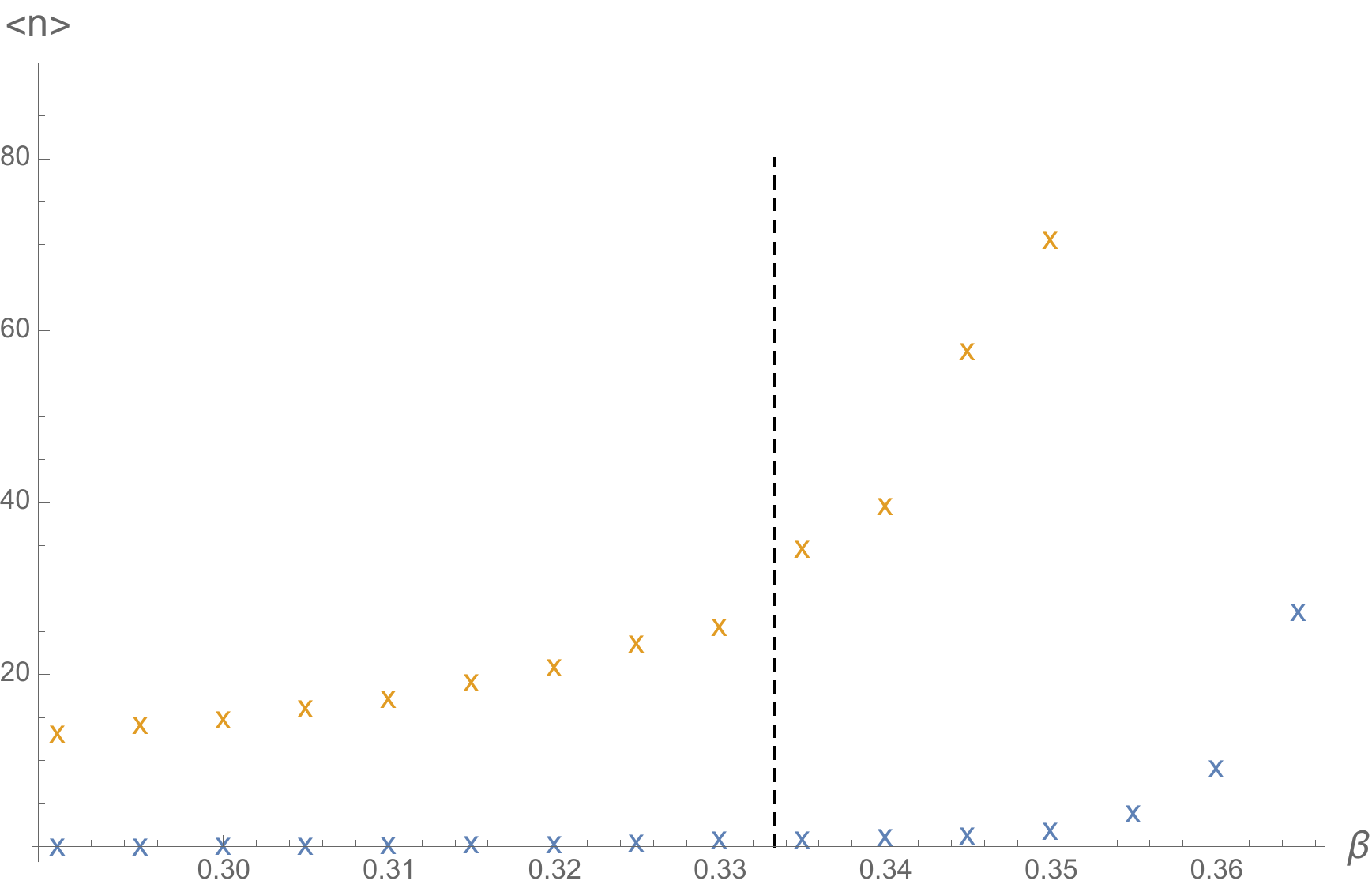} 
\caption{A graph (as in \cite{ERR}), of average word length of ERR random walks plotted against the parameter $\beta$. The orange points come from walks where $\alpha=3$, and the blue points come from walks where $\alpha=0$. The vertical line at $1/3$ marks the expected asymptote.
\label{fig:ERR-BS17}
}
\end{figure}

It is clear that, for this group, the divergence for $\beta>\beta_c$ predicted by the theory is not occurring.  This is further seen in Figure \ref{fig:bs17-distribution}, which shows the progression over time of one of the random walks used to generate Figure \ref{fig:ERR-BS17}.
\begin{figure}
\includegraphics[width=110mm]{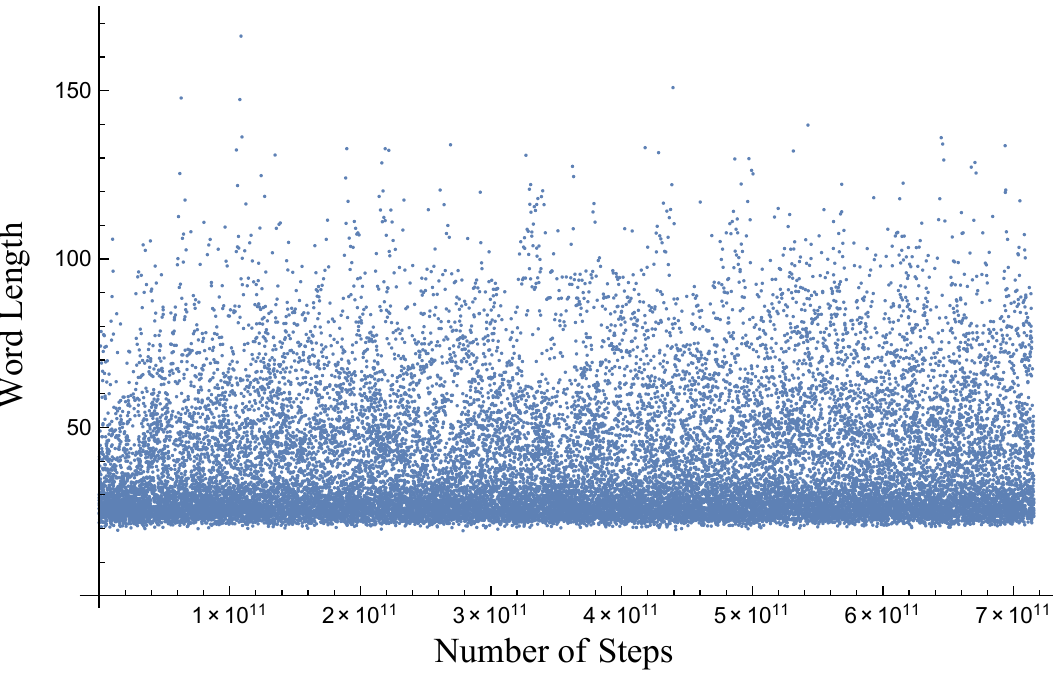} 
\caption{\label{fig:bs17-distribution}
The distribution of ERR random walks on $BS(1,7)$ with $\alpha=3$ and $\beta =0.34$.  
This is a plot of word length against number of steps taken.  The data represents ten ERR random walks overlaid on top of each other.  As can be seen, none of the walks diverged.
Each dot represents the average
word length over 10000 accepted relator insertions.  There is no divergence at this $\beta$ value, even though the group is amenable.
}
\end{figure}
The results in Figure \ref{fig:bs17-distribution} show the 
 word lengths visited for ten ERR random walks (superimposed) performed on $BS(1,7)$, with $\alpha=3$ and $\beta=0.34$.
Since the group has only a single relator, which was successfully inserted into the word 10000 times, it is not an error of the type  identified in Subsection~\ref{subsec:wrong_group}.
The ERR method relies on the divergence of the average word length to identify $\beta_c$, so application of the method in this case will not accurately identify the amenability of 
$BS(1,7)$.

Divergence of the ERR random walk (when $\beta>\beta_c$) relies on the abundance of long trivial words.  For most presentations, at all points in an ERR walk there are always more moves which lengthen the word than shorten it, but the probabilistic selection criteria ensures balance.  More specifically, the parameter $\beta$ imposes a probabilistic barrier which increases exponentially with attempted increase in word length.  
When $\beta >\beta_c$ this exponential cap is insufficient, and the word length diverges.  

Recall that for a given word length $n$ the function $\RR(n)$ quantifies how many reduced-trivial words there are of length similar to $n$.
The results in Table \ref{tab:differentRn} imply that, for many groups, large word lengths must be reached before the asymptotic growth rate is reflected by a local abundance of longer trivial words. 
We have noted in Section \ref{sec:subDomInF} that the 
convergence properties of $BS(1,N)$ in the space of marked groups requires $\RR(n)$ to grow more quickly as $N$ increases.  We now show that the growth rate of 
$\RR(n)$ is sufficient to cause the pathological 
behaviour noted above.

To this end we postulate a hypothetical cogrowth function  for which 
we can explicitly identify and control $\RR(n)$.

\begin{example}
\label{ex:fictional_cogrowth}
Suppose that for some group on two generators and $q>0,\;p\in (0,1)$, the reduced-cogrowth is  known to be exactly $$c_n=3^{n-qn^p}.$$ 
Then $
\limsup_{n\rightarrow\infty} c_n^{1/n}
= 3$
and so the group is amenable. 
It may easily be verfied by the methods outlined in Proposition 
\ref{prop:ProvingRn} that 
$$\RR(n)=\left( 9\log(3)q p2^p   n\right)^{\frac{1}{1-p}}.$$
Note that as $p$ approaches $1$,  the exponent ${\frac{1}{1-p}}$ approaches infinity. This increases both the degree of the polynomial in $n$, and the coefficient $ \left(9\log(3)q p2^p\right)^{\frac{1}{1-p}}$.

Even though we do not know a group presentation with 
precisely this cogrowth function, 
by varying $p$ and $q$ 
this hypothetical example  models  the groups listed in Table~\ref{tab:differentRn}.

Figure \ref{fig:pathological1}
shows the effect of increasing  the parameter $p$ on the ERR random walk 
distribution. Note that this figure is not the output of any computer simulation, rather it models the distributions for an ERR random walk on an amenable group with the hypothetical cogrowth function, for  $\alpha=0,\beta=0.335$ and $q=1$.
\begin{figure}
\includegraphics[width=110mm]{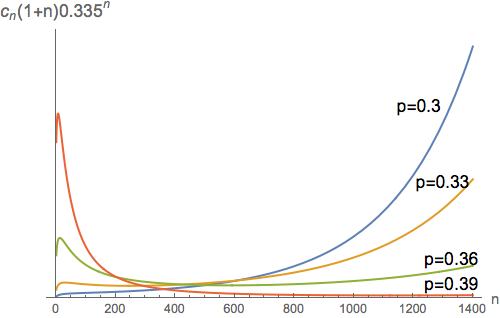}
\caption{
\label{fig:pathological1}
Graphs of $c_n(n+1) 0.335^n$ for $c_n=3^{n-n^p}$. 
}
\end{figure}
Recall that for $\beta<\beta_c$ the theoretical distribution of word lengths visited by the ERR random walk is 
$$\Pr(@n)=\frac{c_n(n+1)^{\alpha+1}\beta^n}{Z}$$ where $Z$ is a normalizing 
constant. 
For $\beta>\beta_c$ the distribution cannot be normalised.  In this case the function $c_n(n+1)^{\alpha+1}\beta^n$ still 
contains information about the behaviour of the walk.  If the 
random walk reaches a word of length $x$ then the relative heights of $c_n(n+1)^{\alpha+1}\beta^n$ either side of $x$ 
describe the relative probabilities of increasing or decreasing
the word length in the next move.

From Figure \ref{fig:pathological1} we see that, for $p=0.3$, 
the slope of $c_n(n+1)^{\alpha+1}\beta^n$ is always positive, so at all word lengths probabilities are uniformly in favour of increasing the word length.
However, as $p$ increases (and the growth rate for $\RR(n)$ increases) a `hump' appears at short word lengths. A random walk for such a group would tend to get stuck in the `hump'.
Indeed, for $p=0.39$ the distribution looks much less like 
a walk diverging towards infinite word lengths and much more like the
distributions for $BS(1,7)$  used to produce Figure~\ref{fig:ERR-BS17}, where the average word length in the ERR walk remained finite.
\end{example}

The distributions in Figure \ref{fig:pathological1} exhibit a 
mechanism which can explain
anomalous behaviour previously observed.
 When $\RR(n)$ increases quickly the ERR random walk may adhere to the behaviour predicted by the theory and simultaneously give anomalous results about the asymptotics of the  cogrowth function.   In this sense if \cite{ERR} contains incorrect answers it is because the original ERR algorithm as it was initially proposed asks the wrong question. The ERR walk does not measure asymptotic properties of the cogrowth function; it provides information about the cogrowth function 
only for word lengths visited by the walk.  This observation forms the basis of Section
\ref{sec:appropriation}.

Note that increasing the parameter $\alpha$  pushes the algorithm towards 
longer word lengths.  Thus, any pathological behaviour caused 
by the growth of $\RR(n)$ could theoretically be 
overcome by increasing $\alpha$.
If $\RR(n)$ is known, then it may be used to calculate
how large words have to get before divergence occurs. A method to do this is outlined by the following example.

Suppose that ERR 
random walks are run on a two generator group with $\beta=0.34$ (as in Figure \ref{fig:bs17-distribution}).  If we eliminate the $\alpha$ term of the stationary distribution (which, being polynomial, becomes insignificant for long word lengths) the divergence 
properties are controlled by the contest between
$0.34^n$ and $c_n$. That is, divergence will occur when 
$c_{2n+2}/c_{2n}>1/0.34^2=3-1/17$; the word length at which divergence will occur is  
$\RR(17)$.   
If this value is known $\alpha$ may be increased 
until the walk visits words of this length.
This process, however, requires specific information about $\RR(n)$ including all scaling constants.  It is hard to imagine a group for which the sub-dominant cogrowth behaviour was known to this level of precision, but dominant cogrowth behaviour (and hence the amenability question for the group) was still unknown.

\subsection{Reliability of the ERR results for Thompson's group $F$}
In Proposition~\ref{prop:connectBStoThompsons} we saw that the $\mathcal{R}$ function for $F$ grows faster than  that of any iterated wreath product of $\Z$'s, and certainly faster than that of any $BS(1,N)$ group.  Since the ERR method fails to predict the amenability of these groups for $N$ as low as $7$, and this behaviour is consistent with the pathological behaviour caused by $\RR$, we conclude that the data encoded in Figure \ref{fig:ERRpaperThompsons} does not imply the non-amenability of $F$, and so the conclusion of the paper \cite{ERR} that $F$ appears to be non-amenable based on this data  is unreliable.

\section{Appropriation of the ERR algorithm \label{sec:appropriation}}

The original implementation of the ERR random walk uses only the 
mean length of words visited in an attempt to estimate asymptotic behaviour of the cogrowth function. 
In this section we show that,
using the full 
distribution of word lengths visited, it is possible to 
estimate specific values
of the cogrowth function.

When doing a long random walk, the probability of arriving at a word of
length $n$ 
can be estimated by multiplying the number of words of that length by the asymptotic probability that the walk ends at a word of this length, $\pi(n)$.
That is, 
\[
\Pr(@n)\approx c_n\pi(n)=c_n\frac{\left(n+1\right)^{\alpha}\beta^{n}}{Z}.
\]
The proportion of the time that the walks spends at words of length
$n$, however, gives us another estimate of $\Pr(@n)$. If we let
$W_n$ be the number of times the walk visits a word of
length $n$ then we have that 
\[
\Pr(@n)\approx\frac{W_n}{Y},
\]
where $Y$ is equal to the length of the walk. From this we obtain
\[
\frac{W_n}{Y}\approx c_n\frac{\left(n+1\right)^{\alpha}\beta^{n}}{Z}.
\]

For two different values, $n$ and $m$, we obtain the result 
\begin{eqnarray*}
\frac{W_m}{W_n} & \approx & \frac{c_m\left(m+1\right)^{\alpha}\beta^{m}}{c_n\left(n+1\right)^{\alpha}\beta^{n}},
\end{eqnarray*}
Thus, 
\begin{equation}
c_m\approx c_n
\frac{W_m}{W_n}
\left(\frac{n+1}{m+1}\right)^{\alpha}\beta^{n-m}.
\label{eqn:cogrowth_estimate}
\end{equation}

Equation~\ref{eqn:cogrowth_estimate} provides a method of estimating the value 
of $c_m$ using some known or previously estimated value of $c_n$ 
and the distribution 
of word lengths visited from an ERR random 
walk.
Let's try  a quick implementation of this for Thompson's group $F$, 
where the first 48 cogrowth terms of which are known \cite{Haagerup}.

We ran an ERR random walk of length exceeding $10^{12}$ steps
on the standard presentation (Equation~\ref{eqn:Fpresentation}) f or $\alpha=3$ and $\beta=0.3$.  The frequency of word 
length visited is shown in Table~\ref{table:singleThompsonsGroupDist}.

\begin{table}
\[\begin{array}{|c|r|}
\hline
n & 	W_n  \\
\hline
0 &  32547326274\\
10 & 56273373521	\\
12 & 	31613690578\\
14 & 	26477475739\\
16 & 	13576713156\\
18 & 	9684082360\\
20 & 	5444250723\\
22 & 	3360907182\\
24 & 	1905434239\\
26 & 	1121735814\\
28 & 	638093341\\
30 & 	367320461\\
32 & 	208025510\\
34 & 	118432982\\
36 & 	65983874\\
38 & 	37210588\\
40 & 	20642387\\
42 & 	11332618\\
44 & 	6243538\\
46 & 3421761\\
48 & 1863477\\
\hline\end{array}\]

\caption{Data collected from an ERR random walk of length $Y=1.8\times 10^{11}$ with  $\alpha=3$ and $\beta=0.3$ on the standard presentation 
for Thompson's group $F$.
\label{table:singleThompsonsGroupDist}
}
\end{table}

\begin{table}[h]
\[\begin{array}{|c|c|c|c|c|c|c|}
\hline
n & 	\text{exact}  & \text{estimate} &  \text{\begin{tabular}{c}percentage\\ error\end{tabular}}\\
\hline
10 & 	20 & 19.9988		 & .006 \\
12 & 	64 & 63.9928		 & .01\\
14 & 	336 & 335.969	 & .01\\
16 & 	1160	 & 1160.23  & .02\\
18 & 	5896	 &  5893.13 & .05\\
20 & 	24652 & 	  24667.2  & .06\\
22 & 	117628	 & 	117588  & .03\\
24 & 	531136	 & 	530650  & .09\\
26 & 	2559552	 & 	2551340 & .3\\
28 & 	12142320	 & 	12116600  & .2\\
30 & 	59416808	 & 	 59353400 & .1\\
32 & 	290915560 & 	290848000 & .02\\
34 & 	1449601452 & 1453990000	  & .3\\
36 & 	7269071976	 & 7206930000  & .8\\
38 & 	36877764000	 & 36583500000 & .8\\
40 & 	1.8848\times 10^{11}	 & 1.8461\times 10^{11}  & 2\\
42 & 	9.7200\times 10^{11}	 &  9.3078\times 10^{11} & 4\\
44 & 	5.0490\times 10^{12}	 &  4.7504\times 10^{12} & 6\\
46 & 	2.6423\times 10^{13}	 &  2.4308\times 10^{13} & 8\\
48 & 	1.3920\times 10^{14}	 &  1.245\times 10^{14} & 10\\
\hline\end{array}\]

\caption{Estimate of the first 48 terms of the cogrowth function
for Thompson's group $F$, constructed from an ERR random walk of $Y=1.8\times 10^{11}$ steps with $\alpha=3$ and $\beta=0.3$. Exact values from \cite{Haagerup}.
\label{tab:unsophisticatedThompsons48}
}
\end{table}

We used Equation~\ref{eqn:cogrowth_estimate} and the data in Table~\ref{table:singleThompsonsGroupDist} to estimate 
$c_{10}$  from $c_0$, and then this estimate was used to estimate
$c_{12}$. (Note that the shortest non-empty trivial words are of length 10.  Since the relators in the standard presentation of $F$ are even in length there are no odd length relators.) 
Using the data and the previous estimate for $c_{n-2}$, estimates were made of the first 48 terms, and these compared to the correct value in 
Table \ref{tab:unsophisticatedThompsons48}.

This implementation of Equation~\ref{eqn:cogrowth_estimate}  may be refined in several ways.  Firstly, in many groups we have exact 
initial values of $c_n$ for more than the trivial result $c_0=1$.  In this 
case these initial values can be used to estimate subsequent terms.  In this paper we are primarily concerned with testing the efficacy of this 
method for determining cogrowth, and so do not make use of such data.

Secondly, in the above implementation the only cogrowth value 
used to estimate $c_n$ was $c_{n-2}$.  Instead, estimates 
for $c_n$ may be made from $c_k$ for any $k<n$.  These estimates
may then be averaged to form an estimate for $c_n$.  Note,
however, that if only one ERR random walk is used, and each of the $c_k$ is itself estimated from previous values of the same distribution there may be issues with interdependence.

This leads naturally to the following refinement --- to obtain several independent estimates for a given cogrowth value
 several ERR random walks can be run with different 
values for the parameters $\alpha $ and $\beta$.

\subsection{The ERR-R algorithm.}
The ERR-R algorithm accepts as input a group presentation and 
the cogrowth value with 
$c_0=1$. As above, recursive application of
Equation~\ref{eqn:cogrowth_estimate} is used to produce estimates for longer
word lengths.  However, in each step previous estimates for a range of $c_n$ are used to produce new estimates.
A detailed analysis of the error incurred with each application of Equation \ref{eqn:cogrowth_estimate} is performed in 
Section \ref{subsec:error_analysis}.  All error bounds which appear in subsequent graphs are constructed using these techniques.  

Unsurprisingly, the errror analysis in Section \ref{subsec:error_analysis} predicts that the largest errors are incurred 
when data is used from the tails of  
random walk distributions.  Ideally then, a separate random walk should be run for each $c_n$, with parameters $\alpha$ and 
$\beta$ chosen so that the sampled word lengths occupy the peaks 
of the distribution. 
 If many estimates are to be made this is computationally infeasible. Instead we performed  ERR random walks
using a range of
$\alpha$ and $\beta$ values, which can be chosen so that all word lengths of interest are visited often.

When estimating $c_m$, one estimate was made from each random walk distribution and from each $c_n,$ $m-100<n<m$. To avoid using the tails of distributions only data points which were greater than 10\% of the max height were used.

Using Equation~\ref{eqn:errorEstimate} each estimate was assigned a weight equal to the inverse of the estimated error.
The final value for $c_m$ was taken as the  weighted average of the estimates, and the error in $c_m$ was taken to be the 
weighted average of the individual error estimates.

Random walk data was obtained as before using the python code of the second author as described in Remark~\ref{rmk:implementation_details}.

\subsection{Application to the examples in Section~\ref{sec:pathological_behaviour}}
Applying the ERR-R algorithm can be used to analyse in more detail the pathological behaviours analysed in this paper.  
Unsurprisingly, for the presentations of the trivial group given in \ref{subsec:wrong_group} which ignore the long relator, the ERR-R estimates for cogrowth values align closely with the three strand braid group.  
For $BS(1,N)$ we can use estimates of initial cogrowth to analyse how $\RR$ increases with $N$.  This is shown, for example in Figure \ref{fig:nthRootBS1N} which exhibits the behaviour predicted by the convergence to $\Z\wr\Z$ in the space of marked groups.
Further analysis of these presentations will appear in \cite{CamPhD}.


\begin{figure}
\includegraphics[width=110mm]{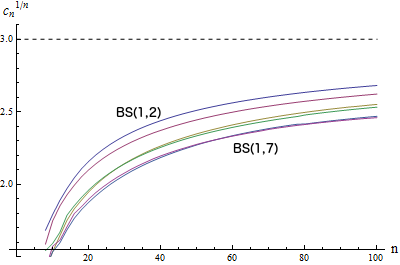} 
\caption{
Estimates for $c_n^{1/n}$ for the groups $BS(1,N)$, $N=2\dots 7$. As $N$ increases the curves takes longer to approach the asymptote.
\label{fig:nthRootBS1N}
}
\end{figure}

\subsection{Application to surface group}
The fundamental group of a surface of genus 2 has presentation $\langle
a,b,c,d\mid [a,b][c,d] 
\rangle$.
The cogrowth of this group has received a lot of attention, and good upper and lower bounds are known for the asymptotic rate of growth \cite{Gouezel,Nag}. 

ERR random walks were run on this surface group with $\alpha=3,\;30,\;300$
and $\beta=0.281,\;0.286,\;0.291,\dots,0.351$.  Estimates were made for $c_n$ as well as the error $\Delta c_n$.
The resultant upper and lower bounds for $c_n^{1/n}$ are shown in 
Figure \ref{fig:nthRootSurface}.

\begin{figure}
\includegraphics[width=110mm]{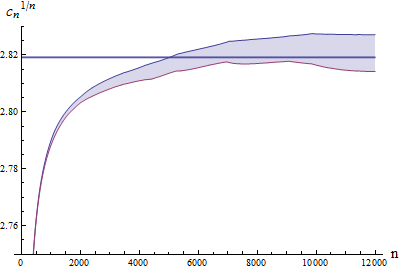} 
\caption{Upper and lower bounds for the $n$-th root of the cogrowth function
for the fundamental group of a surface of genus 2 as calculated from ERR random walks.  The horizontal lines (indistinguishable at this scale) identify the known upper and lower bounds. Note that after 12000 recursive applications of Equation~\ref{eqn:cogrowth_estimate} the error in the $n$-th root is still only approximately 0.01. \label{fig:nthRootSurface}} 
\end{figure}

\subsection{Application to Thompson's group $F$}
We now apply the more sophisticated implementation of the method to $F$. Recall that we can compare the first 48 values with exact values obtained by Haagerup {\em et al.}. Our method allows us to go much further than this though, which we do. 

ERR random walks were run on $F$ with $\alpha=3,13,23,33,53,63$ and $\beta=0.28,0.29,\dots 0.37$.
Collection of experimental data is ongoing.  Table \ref{tab:sophisticatedThompsons48} shows comparisons between estimates for $c_n^{1/n}$ and the actual values, for $n\leq 48$, as well as the estimates for the error obtained from the experimental data.

\begin{table}
\[\begin{array}{|c|r|r|c|c|c|c|}

\hline
n & 	\text{exact}  & \text{estimate} &  \text{\begin{tabular}{c} error\ (\%)\end{tabular}}&  \text{\begin{tabular}{c}predicted\\ error\ (\%)\end{tabular}}\\
\hline
10 & 	20 & 19.9996	 & 0.002 & .03\\
12 & 	64 & 	63.9981 & 0.003 & 0.06\\
14 & 	336 & 	335.999 & 0.0002& 0.07\\
16 & 	1160	 &  1159.96 & 0.003& 0.1\\
18 & 	5896	 &  5895.98 & 0.0003& 0.1\\
20 & 	24652 & 	 24653.1 & 0.005& 0.1\\
22 & 	117628	 & 117625  & 0.003& 0.2\\
24 & 	531136	 &  531098 & 0.007& 0.2\\
26 & 	2559552	 & 2558950	 & 0.02& 0.2\\
28 & 	12142320	 & 	12138200  & 0.03& 0.3\\
30 & 	59416808	 & 	 59408300 & 0.01& 0.3\\
32 & 	290915560 & 	 290861000 & 0.02& 0.3\\
34 & 	1449601452 &	  1449260000 & 0.02& 0.3\\
36 & 	7269071976	 & 7268550000 & 0.007& 0.4\\
38 & 	36877764000	 & 36876700000 & 0.003& 0.5\\
40 & 	1.8848\times 10^{11}	 & 1.88491 \times 10^{11}  & 0.003& 0.5\\
42 & 	9.7200\times 10^{11}	 & 9.7205 \times 10^{11} & 0.005& 0.5\\
44 & 	5.0490\times 10^{12}	 & 5.05097\times 10^{12} & 0.04& 0.6\\
46 & 	2.6423\times 10^{13}	 & 2.64353\times 10^{13} & 0.05& 0.6\\
48 & 	1.3920\times 10^{14}	 & 1.39246\times 10^{14} & 0.03& 0.7\\
\hline\end{array}\]

\caption{Estimate of the first 48 terms of the cogrowth function
for Thompson's group $F$, constructed from 60 ERR random walks. Exact values from \cite{Haagerup}.
\label{tab:sophisticatedThompsons48}
}
\end{table}

\begin{rmk}
Table \ref{tab:sophisticatedThompsons48} shows a marked increase in the  degree of accuracy of the estimates over those of Table \ref{tab:unsophisticatedThompsons48}.  This suggests the method of using multiple distributions and weighted averages is effective. Note that there are approximately $10^{12}$ trivial words of length 48 so the walks could not possibly have visited each one.  The sample of words visited by the walk seem to reflect the space as a whole reasonably accurately.
\end{rmk}

Figure \ref{fig:thompsonsNthRoot}  shows our estimates for upper and lower bounds of $c_n^{1/n}$ for $n\leq 2000$.

\begin{figure}
\includegraphics[width=110mm]{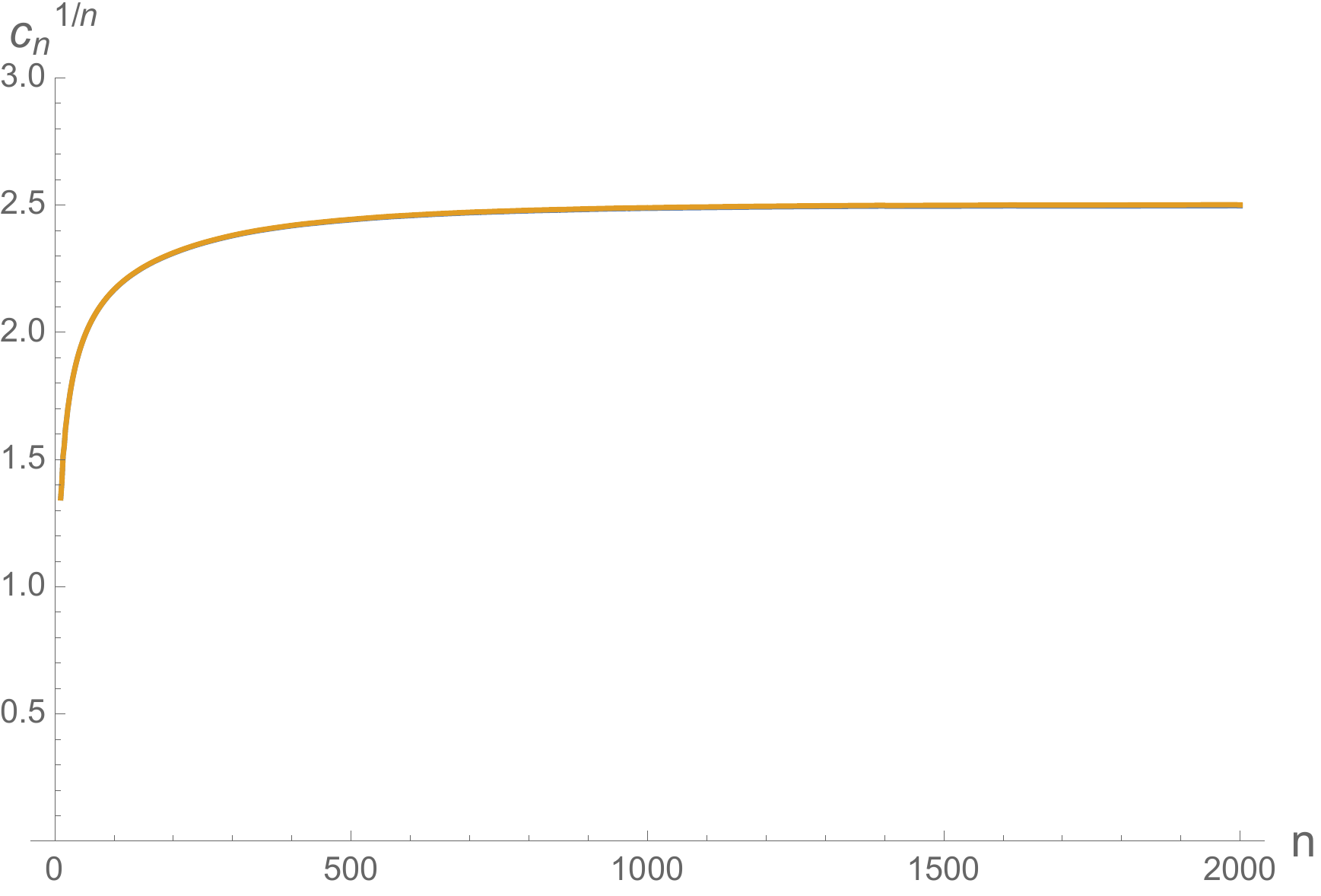} 
\caption{
Estimates of $c_n^{1/n}$ for Thompsons group $F$ for $n\leq 2000$, using the ERR-R method.   
The figure
includes upper and lower bounds, but at this scale the estimated error 
is to small for the bounds to be distinguished.
\label{fig:thompsonsNthRoot}}
\end{figure}

\subsection{Error analysis}\label{subsec:error_analysis}

Here we identify a method by which error in cogrowth estimates my be estimated.  We stress that this is a statistical measurement of error, rather than theoretical.  

Recall Equation~\ref{eqn:cogrowth_estimate}.
Suppose that $c_n$ is known up to $\pm\Delta c_n$,
and that the error in the measurements $W_m$ and $W_n$ are 
$\pm\Delta W_m$ and $\pm\Delta W_n$ respectively.  Then, 
from elementary calculus, the error in 
$c_m$ is given by 

\begin{align}
\nonumber
\Delta c_m
\approx &\frac{W_m}{W_n}
\left(\frac{n+1}{m+1}\right)^{\alpha}\beta^{n-m} \Delta c_n\\
\nonumber
&+ \frac{c_n}{W_n}
\left(\frac{n+1}{m+1}\right)^{\alpha}\beta^{n-m} \Delta W_m\\
\nonumber
&+ c_n\frac{W_m}{W_n^2}
\left(\frac{n+1}{m+1}\right)^{\alpha}\beta^{n-m} \Delta W_n\\
\nonumber
=&c\left(n\right)
\frac{W_m}{W_n}
\left(\frac{n+1}{m+1}\right)^{\alpha}\beta^{n-m}
\left( 
\frac{\Delta c_n}{c_n}
+\frac{\Delta W_m}{W_m}
+\frac{\Delta W_n}{W_n}
\right)\\
\approx&c_m
\left( 
\frac{\Delta c_n}{c_n}
+\frac{\Delta W_m}{W_m}
+\frac{\Delta W_n}{W_n}
\right).\label{eqn:errorEstimate}
\end{align}

Hence the proportional error in the estimate of $c_m$ is 
approximately equal to the sum of the proportional errors in $c_n,\,W_m$ and $W_n$.  It is clear from this that if Equation~\ref{eqn:cogrowth_estimate} is used recursively (building new
estimates based on previously estimated cogrowth values) the proportional 
error in $c_n$ is certain to increase. Note, the factor controlling the rate of growth in the proportional error of 
estimates is the proportional error in $\Delta W_n$.  If this 
is constant as $n$ increases the proportional error in $c_n$ will grow linearly with $n$.

To calculate useful error margins for $c_n$ it is necessary to quantify $\Delta W_n$.  Here we employ the same method 
used in the ERR paper; walks are split into $M$ segments and
the number of times the walk visits words of length $n$ is recorded
for each segment. Let $x_{i,n}$ denote the number of times the walk
visited words of length $n$ in the $i$th segment. Then $W_n$ is taken to be the average of $x_{i,n}$ for $i=1\dots M$ and the error in 
$W_n$ is calculated from the statistical variance of these values,

\begin{equation}\label{eqn:errorInWn}
\Delta W_n=\sqrt{\frac{\var\lbrace x_{i,n}\rbrace_{1\leq i\leq M}}{M-1}}.
\end{equation}

\begin{example}
Equations \ref{eqn:errorEstimate} and \ref{eqn:errorInWn}  were used to produce the estimates of the error in the estimates contained in 
Table~\ref{tab:sophisticatedThompsons48}. 
Note that the estimated error is much larger then the actual error.
\end{example}

\subsection{Error in the $n$-th root of $c_n$}
We have noted that recursive uses of Equation~\ref{eqn:cogrowth_estimate} will result in an increasing proportional error in $c_n$. 
However, it is the $n$-th root of $c_n$ which reflects the amenability of a group.  Let 
$\gamma_n=c_n^{1/n}$ and 
$\Delta \gamma_n$ denote the error of the estimate for $\gamma_n$.
Once again from elementary calculus we obtain that for a 
given $n$

\begin{align}
\nonumber
\Delta \gamma_n
&\approx\frac{1}{n}c_n^{\frac{1}{n}-1}\Delta c_n\\
\nonumber
&=\frac{1}{n}c_n^{\frac{1}{n}}\frac{\Delta c_n}{c_n}\\
\nonumber
&=\gamma_n\frac{1}{n}\frac{\Delta c_n}{c_n}\\
\text{and so }\frac{\Delta \gamma_n}{\gamma_n} &\approx\frac{1}{n}\frac{\Delta c_n}{c_n}.\label{eqn:errorInNthRoot}
\end{align}

Thus, if $\frac{\Delta c_n}{ c_n}$ increases at most linearly, $\frac{\Delta \gamma_n}{\gamma_n}$ can be expected to remain constant.  

The values for $c_n$ grow exponentially, so a linearly increasing proportional error in $c_n$ corresponds with a massive increase in the absolute error in $c_n$.  In contrast, $\gamma_n$ approaches a constant, so the proportional error depends linearly on the absolute error. 
Thus it is not surprising that our  experimental results show that even when the error in 
cogrowth estimates grows large, the error in 
the $n$-th root grows very slowly.

\section{Conclusion}

Several ideas emerge from this study.
Firstly, researchers performing experimental mathematics to determine the amenability of a group need to take care that their algorithm is not susceptible to interference from sub-dominant behaviours. For the reduced-cogrowth function the sub-dominant behaviour is identified by $\RR$.   Amenability is  an asymptotic property, and
the interference of sub-dominant behaviours on experimental algorithms can be subtle and nuanced.  In particular, we have shown that, if Thompson's group $F$ is amenable, its function $\RR$ grows faster than any polynomial.  This implies that the prediction of non-amenability of $F$ in \cite{ERR} is unreliable.

We have also shown that, despite potential inaccuracies in estimates of asymptotics, the ERR-R method can produce accurate results for initial cogrowth values.
These
are interesting in their own right.  Indeed, if Thompson's group is not amenable, then its $\RR$ function need not be super-polynomial and results from experimental methods might well inform the construction of conjectures regarding cogrowth.

In this context the original benefits of the ERR algorithm still stand:
it requires no group theoretic computational software, no solution to the word problem, and remains a computationally inexpensive way to quickly gain insight into the cogrowth function of a finitely presented  group.

\section*{Acknowledgements}

The authors wish to thank  Andrew Rechnitzer and 
Andrew Elvey-Price  
  for helpful feedback on this work.

\bibliography{refs} \bibliographystyle{plain}

\end{document}